\documentclass[12pt]{amsart}
\usepackage{fullpage}
\usepackage[cp1250]{inputenc}
\usepackage[T1]{fontenc}
\usepackage{latexsym}
\usepackage{amssymb,amsmath}
\usepackage{dsfont}
\usepackage{pb-diagram}
\usepackage{textcomp}
\usepackage{stmaryrd}
\usepackage{mathrsfs}
\usepackage{amsmath}
\usepackage{amssymb}
\usepackage{multicol}

\newcommand{\CC}{\mathbb C}

\newcommand{\PP}{\mathbb P}

\newcommand{\RR}{\mathbb R}

\newcommand{\ZZ}{\mathbb Z}
\newcommand{\cA}{\mathcal A}

\newcommand{\cM}{\mathcal M}

\newcommand{\cO}{\mathcal O}

\newcommand{\Sum}{\sum\limits}

\newcommand{\tensor}{\otimes}
\newcommand{\emb}{\hookrightarrow}

\renewcommand{\Bar}{\overline}

\newcommand{\imic}{\cong}

\newcommand{\ratmap}{\dasharrow}

\newcommand{\Km}{\mathop{\mathrm {Km}}\nolimits}

\DeclareMathOperator{\Ima}{Im}

\DeclareMathOperator{\diag}{diag}
\DeclareMathOperator{\NS}{NS}
\DeclareMathOperator{\Is}{Is}
\DeclareMathOperator{\Hom}{Hom}

\DeclareMathOperator{\mult}{mult}

\DeclareMathOperator{\Aut}{Aut}

\newcommand\To{\longrightarrow}
\newcommand\fD{\mathfrak{D}}

\newcommand\SH{\mathfrak{h}}

\theoremstyle{plain}
\newtheorem{theorem}{Theorem}[section]
\newtheorem{definition}[theorem]{Definition}
\newtheorem{lemma}[theorem]{Lemma}
\newtheorem{corollary}[theorem]{Corollary}
\newtheorem{proposition}[theorem]{Proposition}

\theoremstyle{remark}
\newtheorem{remark}{Remark}

\newcommand{\Th}[2]{{\theta\genfrac{[}{]}{0pt}{1}{#1}{#2}}}

\begin{document}
\title{Hyperelliptic genus 4 curves on abelian surfaces}
\author{Pawe\l{} Bor\'owka, G.K. Sankaran}
\address{P. Bor\'owka \\ Institute of Mathematics, Jagiellonian
University in Krak\'ow, ul.\ prof Stanis{\l}awa {\L}ojasiewicza~6, 30-348 Krak\'ow, Poland}
\email{Pawel.Borowka@uj.edu.pl}
\address{G.K. Sankaran \\
Department of Mathematical Sciences, University of Bath, Bath BA2~7AY,
England}
\email{G.K.Sankaran@bath.ac.uk}

\begin{abstract}

We study smooth curves on abelian surfaces, especially for genus~$4$,
when the complementary subvariety in the Jacobian is also a
surface. We show that up to translation there is exactly one genus~$4$
hyperelliptic curve on a general $(1,3)$-polarised abelian surface. We
investigate these curves and show that their Jacobians contain a
surface and its dual as complementary abelian subvarieties.
\end{abstract}
\maketitle

\section{Introduction}\label{sect:intro}

A smooth curve $C$ in an abelian surface~$A$ over the complex numbers
cannot be of genus~$0$, and if it is of genus~$1$ then $A$ is
isogenous to a product $C\times C'$ of genus~$1$ curves. A smooth
genus~$2$ curve can be embedded only in one abelian surface, namely as
the theta divisor in its Jacobian. For curves of higher genus, we do
not know much. There is one well known example: \'etale degree~$n$
cyclic covers of genus~$2$ curves, which provide curves embedded in
$(1,n)$-polarised surfaces. On the other hand, a general section of a
polarising line bundle on a $(d_1,d_2)$-polarised surface defines a
smooth curve of genus $g=1+d_1d_2$. In general we know little about
the geometry of these curves. For $(1,2)$ polarisation,
W. Barth~\cite{B} shows that all curves in the linear series of a
polarising line bundle are double covers of elliptic curves branched
in $4$ points. For $E\times E$ with a $(1,3)$ polarisation,
Ch.~Birkenhake and H.~Lange \cite{BL2} listed the types of curves
in the linear series.

In this paper, we describe a distinguished family of curves in
$(1,3)$-polarised abelian surfaces. As the construction is similar to
the construction of the theta divisor on a principally polarised
surface and their properties are similar, we call them $(1,3)$ theta
divisors.

The paper is organised in the following way.  Section~\ref{sect:prelim}
first provides some notation and background. Then, assuming that
$f_C\colon C \to A$ is an embedding of a smooth complex curve of
genus $g>2$ in an abelian surface, we show (Lemma~\ref{lem:kerf}) that
$f_C$ induces an exact sequence with connected kernel
$$
0\To K\To JC\To A\To 0.
$$ 
Thus $JC$ is non-simple and contains $K$ and $\hat{A}$ as
complementary abelian subvarieties. Moreover, the restriction to
$\hat{A}$ of the natural principal polarisation on $JC$ is dual to the
polarisation $c_1(\cO_A(f_C(C)))$ on $A$.

Section~\ref{sect:g=4} focuses on the case $g=4$, where there is
additional symmetry. Firstly, $g=1+d_1d_2$, so the type of
polarisation has to be $(1,3)$, and secondly $\dim(JC)=4$, so $K$ is
also a $(1,3)$-polarised surface. One main result of the paper
(Theorem~\ref{thm:jednahiper}) is that a general $(1,3)$-polarised
abelian surface contains, up to translation, exactly one smooth
hyperelliptic curve of genus~$4$. This result is consistent with the
number~$9$ of hyperelliptic genus~$4$ curves on a $(1,3)$-polarised
surface given by the counting function $h^{A,FLS}_{g,\beta}$ defined
in \cite{BOPY}: see Remark~\ref{rem:BOPY}.

Moreover, in this case of $g=4$ we show that $K=A$, and indeed we
characterise Jacobians of such curves as hyperelliptic Jacobians that
contain a surface $A$ and $\hat{A}$ as complementary abelian
subvarieties: see Theorem~\ref{thm:dualpair}.

Some results of this paper are contained in the PhD thesis of the
first author.

\section{Preliminaries}\label{sect:prelim}
\subsection{Theta functions}\label{subsect:thetafns}
We recall very briefly some basic facts about theta functions,
characteristics and theta constants. For simplicity of notation we
shall immediately restrict our attention to $(1,3)$-polarised surfaces
$(A,H)$.  A line bundle denoted $L$ will always be a \emph{polarising
  line bundle}, i.e.\ $c_1(L)=H$.  For generalisations (to higher
degree and higher dimension), proofs and details, we refer
to~\cite{BL}.  For $Z\in\SH_2$ (the Siegel upper half-plane), we
denote by $A_Z$ the complex torus $\CC^2/(Z\ZZ^2+\diag(1,3)\ZZ^2)$,
always with the polarisation $H\in \NS(A_Z)$ whose Riemann form, also
denoted by $H$, is given by $H=\Ima(Z)^{-1}$. We take the standard
decomposition $\CC^2=Z\RR^2\oplus\diag(1,3)\RR^2$.

Because we use both Riemann canonical and Riemann classical theta
functions, we will have to abuse notation in the following way.  For a
canonical theta function we think of its characteristic as a point
$c\in\CC^2$ (a complex characteristic) and write $c=c_1+c_2$, where
$c_1\in Z\RR^2$ and $c_2\in \RR^2$. For a classical theta function its
characteristic will be a pair $(c_1,c_2)$ (a real characteristic),
where $c_1,\,c_2\in\RR^2$, and we put $c=Zc_1+\diag(1,3)c_2\in\CC^2$.

Any $2$-torsion point $x\in A_Z[2]$ can be written as the image of
$x=Zc_1+\diag(1,3)c_2$ for some $c_1,\,c_2\in\frac{1}{2}\ZZ^2$.

\begin{definition}\label{def:even2torsion}
Define $e_*(c_1,c_2)=\exp(4\pi
i\ ^tc_1c_2)\in\{\pm 1\}$.
A $2$-torsion point $x=Zc_1+\diag(1,3)c_2$ is called \emph{even} or \emph{odd}
depending on the parity of $e_*(c_1,c_2)$.
\end{definition}
A simple computation shows that there are exactly ten even and six odd
$2$-torsion points.

On $A_Z$ there exists a unique polarising line bundle of
characteristic~$0$, which we denote $L_0$. Any other polarising
line bundle is of the form $L=t^*_cL_0$ for some $c\in A_Z$ unique up to
$K(L)=\{x\in A_Z\colon t^*_x(L)\cong L\}\cong\ZZ/3\ZZ$.

The \emph{classical Riemann theta function} of real characteristic
$(c_1,c_2)$ is (see~\cite[Section~3.2]{BL})
\begin{equation}\label{eq:classicaltheta}
\Th{c_1}{c_2}(v,Z) =
\sum_{l\in\ZZ^2}\exp\Big(\pi i
\ ^t(l+c_1)Z(l+c_1)+2\pi i \ ^t(v+c_2)(l+c_1)\Big).
\end{equation}
The \emph{canonical Riemann theta function} of complex characteristic $c$ is
\begin{equation}\label{eq:canonicaltheta}
\begin{split}
\theta^c(v)=\exp&\Big(-\pi
H(v,c)-\frac{\pi}{2}H(c,c)+\frac{\pi}{2}B(v+c,v+c)\Big)\\
&\cdot\Sum_{\lambda\in Z\ZZ^2}\exp\Big(\pi
(H-B)(v+c,\lambda)-\frac{\pi}{2}(H-B)(\lambda,\lambda)\Big).
\end{split}
\end{equation}
where $H$ is a Riemann form and $B$ is the bilinear extension of
$H|_{\ZZ^2}$. For any $\eta\in K(L)_1=K(L)|_{Z\RR^2}$, we define
\[
\theta^c_\eta=a_L(\eta, \cdot)^{-1}\theta^c(\cdot+\eta)
\] 
For details, again see~\cite[Section 3.2]{BL}.

In the case of $(1,3)$-polarised abelian surfaces, $K(L)_1$ is generated
by $\omega = Z(0,\frac{1}{3})$ which, when we write classical theta
functions, becomes $\omega=(0, \frac{1}{3})$ by the above convention.
Then $H^0(L_0)$ can be identified with a space of classical theta
functions and $\Th{0}{0},\,\Th{\omega}{0},\,\Th{-\omega}{0}$ form a
basis. It can also be identified with a space of
canonical theta functions and in this case
$\theta^0_0,\,\theta^0_{\omega},\,\theta^0_{-\omega}$ form a basis.

\begin{definition}\label{def:symmetricbundle}
A polarising line bundle on $A=A_Z$ is said to be \emph{symmetric} if
$(-1)^*L\cong L$. 
\end{definition}

For odd degree, the characteristic of a symmetric line bundle $L$ can
be chosen uniquely to be a $2$-torsion point on $A$.  If $L$ is
symmetric, then $(-1)$ acts on $H^0(L)$, acting on (either kind of) theta
function by $(-1)^*\theta(z)=\theta(-z)$. The Inverse
Formula~\cite[Inverse Formula 4.6.4]{BL} gives this action in terms of
the basis of canonical theta functions.
\begin{proposition}[BL, Inverse Formula 4.6.4]
Let $L$ be a symmetric line bundle of characteristic $c=c_1+c_2$. Then
for $\eta\in K(L)_1$ we have
\[
(-1)^*\theta^c_\eta= \exp\big(4\pi i\Ima H(\eta+c_1,c_2)\big)\theta^c_{-\eta-2c_1}.
\]
\end{proposition}
We denote the $(\pm 1)$-eigenspaces by $H^0(L)_{\pm}$ and call the
corresponding theta functions even and odd theta functions.

\subsubsection{Divisors}
Let $\fD$ be an effective ample divisor on an abelian surface $A$. We
say that $\fD$ is symmetric if $(-1)^*\fD=\fD$. By $\mult_x\fD$ we
denote the multiplicity of $\fD$ in $x\in A$. Then $\fD$ is called
even or odd depending on the parity of $\mult_0\fD$.

We need to understand the behaviour of a symmetric divisor in $A[2]$.
\begin{definition}\label{def:A2pm}
Define
\begin{align*}
A[2]_{\fD}^+&=\{c\in A[2]\colon \mult_c(\fD)\equiv 0 \mod 2\},\\
A[2]_{\fD}^-&=\{c\in A[2]\colon \mult_c(\fD)\equiv 1 \mod 2\}.
\end{align*}
\end{definition}
\cite[Exercise 4.12.14]{BL} gives us the following.
\begin{proposition}\label{prop:tenpoints}
If $\fD$ is an effective symmetric divisor on $(A,H)$ such that
$\cO(\fD)=L_0$, of type $(1,3)$, then
\[
\# A[2]_{\fD}^\pm=
\begin{cases}
      2(4\pm 1) \ \text{if $\fD$ is even}\\ 
      2(4\mp 1)\ \text{ if $\fD$ is odd.}
\end{cases}
\]
Moreover, there are ten (respectively six) symmetric bundles $L$ such
that $\# A[2]_{\fD}^-=6$ (respectively $\# A[2]_{\fD}^-=10$) 
 for all even symmetric divisors with $\cO(\fD)=L$.
\end{proposition}

\subsection{Embedded curves}\label{subsect:curves}

Suppose that $f_C\colon C\to A$ is an embedding of a smooth curve of
genus $g>1$ in an abelian surface.  Let $(JC,\Theta)$ denote the
polarised Jacobian of $C$.  Without loss of generality, we can choose
a point $O\in C$ such that $f_C(O)=0$. Then, by the universal property
of Jacobians, we have the following diagram:
\begin{equation}\label{Diag1}
  \begin{diagram}
\dgARROWLENGTH=2.5em
    \node{C}\arrow[4]{e,t}{f_C}\arrow[2]{se,r}{\alpha_O}
    \node[4]{A}\\[2]\node[3]{JC}
\arrow [2]{ne,r}{f}
    \\[2]
    \node[1]{K^0}\arrow[2]{ne,r}{k}  
  \end{diagram}
\end{equation}
where $\alpha_O$ is the Abel-Jacobi map, $f$ is the canonical
homomorphism defined by the universal property and $K^0$ is the
identity component of the kernel of $f$. 

The image $f_C(C)$ generates $A$, so $f$ must be surjective, and hence
$K^0$ is an abelian subvariety of dimension $g-2$. The following lemma
tells us that in fact $K^0=\ker(f)$.

\begin{lemma}\label{lem:kerf}
The kernel of $f$ is connected. Hence, dualising the exact sequence
$$
0\To \ker(f)=K^0\stackrel{k}{\To} JC \stackrel{f}{\To} A\To 0
$$ 
we get an embedding $\hat f\colon \hat A \emb JC$.
\end{lemma}
\begin{proof}
The kernel of $f$ is a reduced effective $2$-cycle in $JC$ consisting
of a finitely many connected components $K^0,\dots,K^t$, each $K^i$
being a copy (a translate) of the identity component $K^0$, and in
particular numerically equivalent to $K^0$.

Each $P\in A$ defines Abel-Jacobi map $\alpha_P\colon C \to JC$, given
by $\alpha_P(Q)=\cO_C(Q-P)$. We define the difference map
$\delta\colon C\times C\to JC$ by
\[
\delta(P,Q)\mapsto \cO(P-Q)=\alpha_Q(P)\in JC.
\]
We claim that the image $\delta(C\times C)$ is an effective $2$-cycle.
Certainly the image is effective and irreducible, and has dimension at
least~$1$ because it contains $\delta(C\times\{O\})=\alpha_O(C)\imic
C$. Indeed it contains $\alpha_P(C)$ for every $P\in C$, so if it is
of dimension~$1$ then $\alpha_P(C)=\alpha_O(C)$ for every $P\in
C$. But then the isomorphism
\[
\tensor \cO_C(O-P)\colon \alpha_O(C)\To \alpha_P(C),
\]
which is the restriction of $t_{[\cO(O-P)]}$, becomes an automorphism
of $\alpha_O(C)$; moreover, the group of automorphisms acts
transitively on $\alpha_O(C)$ because $\alpha_O(P)=-\alpha_P(O)$ is sent to
$\alpha_O(Q)$ by  $t_{[\cO(Q-P)]}$. This is impossible, because a
curve of genus $g>1$ does not have transitive automorphism group.

Next, we show that $\delta(C\times C)\cap\ker{f}=\{0\}$, and in
particular is connected.  One inclusion is obvious. For the other,
choose a point $\cO(P-Q) \in \delta(C\times C)\cap \ker(f)$. It is
obviously in the image of $\alpha_Q$.  Now, from the universal
property of the Jacobian, $f$ makes the following diagram commutative
$$
\begin{diagram}
\dgARROWLENGTH=2em
\node{C}\arrow[2]{e,t}{f_C}\arrow[2]{s,r}{\alpha_Q}    \node[2]{A}\\
[2]\node{JC}\arrow[2]{e,t}{f}
\node[2]{A}\arrow[2]{n,r}{t_{f(Q)}}
  \end{diagram}
$$ 
But $C$ is embedded in $A$ and in $JC$ so $f|_{\alpha_Q(C)}$ has to be
injective. Hence $\ker(f)\cap\alpha_Q(C)=\{0\}$ and therefore
$\cO(P-Q)=0$.

That means that $\delta(C\times C)$ has non-empty intersection with
exactly one of the connected components of $\ker(f)$; but the
connected components are all numerically equivalent, so there can only
be one of them.
\end{proof}

\subsubsection{Polarisations}
Because $g>1$, the bundle $\cO_A(f_C(C))$ is ample, of some type
$(d_1,d_2)$ with $g=1+d_1d_2$.  Recall that the dual abelian variety
$\hat A$ of a polarised abelian variety $(A,H)$ carries a uniquely
defined dual polarisation, denoted~$\hat H$~\cite[Prop. 14.1.1]{BL}.

\begin{proposition}\label{prop:dualtheta} 
In Lemma~\ref{lem:kerf}, we have
$\hat{f}^*\Theta\equiv \widehat{\cO(f(C))}$ as classes in
$\NS(\hat{A})$.
\end{proposition}
\begin{proof}
This is a direct application of the following proposition.
\end{proof}
\begin{proposition}\cite[Proposition~4.3]{BL2}\label{prop:indpol}
Let $C$ be a smooth curve and $(JC,\Theta)$ its Jacobian. Let $(A,H)$
be a polarised abelian surface and suppose $f_C\colon C\to A$ is a
morphism and $f\colon JC\To A$ is the canonical homomorphism defined by the
universal property. Then the following are equivalent:
\begin{enumerate}
\item $\hat{f}^*\Theta\equiv \hat{H}$;
\item $(f_C)_*[C]=H$ in $H^2(A,\ZZ)$.
\end{enumerate}
\end{proposition}
\begin{corollary}\label{cor:genus2}
If $g(C)=2$, then $f$ is an isomorphism, so $(A,
\cO(f(C)))=(JC,\Theta)$ is the Jacobian of $C$.
\end{corollary}
Proposition \ref{prop:indpol} allows us to invert the construction in
the following way.
\begin{proposition}\label{prop:recover}
Let $C$ be a smooth curve of genus $g$ whose Jacobian contains an
abelian surface, denoted $\hat{A}$. If $\Theta|_{\hat A}$ is of type
$(d_1,d_2)$ and $d_1d_2=g-1$, then $C$ can be embedded in the dual
surface $A=\hat{\hat A}$ and we recover Diagram~\eqref{Diag1}.
\end{proposition}

\begin{proof}
Let $\hat{f}\colon\hat{A}\To JC$ be the inclusion. We can dualise it
to get a map $f\colon JC\To A$. Choose a point $O\in C$ and consider
the Abel-Jacobi map $\alpha_O$. Then $f_C=f|_{\alpha_O(C)}\colon C\to
A$ is a morphism and by construction and
Proposition~\ref{prop:indpol}, $c_1(\cO_A(f_C(C)))$ has to be of type
$(d_1, d_2)$. As the arithmetic genus of the image equals $1+d_1d_2=g$
we get that $f_C$ has to be an isomorphism onto its image, and
therefore an embedding.
\end{proof}
As in \cite{Bor} we denote by $\Is^g_D$ the locus in $\cA_g$ of
principally polarised abelian varieties of dimension~$g$ containing an
abelian subvariety of dimension $k$ on which the restricted
polarisation has type $D=(d_1,\ldots, d_k)$. The restriction
$k\leq\frac{g}{2}$, which is imposed in \cite{Bor} to exclude empty
cases, is not needed here, but note that $\Is^3_{(1,2)}=\Is^3_{(2 )}$
because the complementary abelian subvariety to a $(1,2)$-polarised
surface is an elliptic curve with restricted polarisation of
type~$(2)$.

\begin{theorem}
For every $g\geq 3$ and any $D=(d_1,d_2)$ such that $g=1+d_1d_2$, there
is a $(g+1)$-dimensional family of non-isomorphic smooth curves of
genus~$g$ that have non-simple Jacobians belonging to
$\Is^g_{(d_1,d_2)}$.
\end{theorem}
\begin{proof}
For a general $(d_1, d_2)$-polarised surface $(A,H)$, we have
$h^0(A, L)=d_1d_2$ and the zero set of a general section is a smooth
curve of genus $g=1+d_1d_2$. The moduli space $\cA_D$ of
$(d_1,d_2)$-polarised abelian surfaces is
of dimension~$3$. We claim that a given genus~$g$ curve has only
finitely many embeddings of this kind, up to translation. More
precisely, if $O\in C$ is a base point, the set
$$
\Phi_D=\left\{f_C\colon (C,O)\to (B,0)\mid f_c\,\text{an
  embedding},\ (B,H_B)=(B,[\cO(f(C))])\in \cA_D\right\}
$$
is finite. This follows because Propositions~\ref{prop:indpol}
and~\ref{prop:recover} give a bijection between $\Phi_D$ and
$$
\Psi_D=\left\{\hat{f}\in\Hom(\hat{B},JC)\mid
\hat{f}\,\text{injective},\ \hat{f}^*\Theta=\hat{H}_B\right\}.
$$
But $JC$ contains only finitely many abelian surfaces such that
$\Theta$ restricts to a polarisation of type~$(d_1,d_2)$, and
$\Aut(\hat{B})$ is also finite, so $\Psi_D$ is finite.

Hence the family of curves that arises in this way
is of dimension $d_1d_2-1+3=g+1$.
\end{proof}

In genus~$3$, we have recovered a result by W.~Barth
\cite[Prop~1.8]{B}.
\begin{corollary}\label{cor:barthbiellgenus3}
The following conditions on a smooth genus~$3$ curve $C$ are
equivalent:
\begin{itemize}
\item $C$ can be embedded in a $(1,2)$-polarised abelian surface
\item $JC\in\Is^3_{(2)}=\Is^3_{(1,2)}$
\item $C$ is bielliptic, i.e.\ a double cover of an elliptic curve
  branched in four points. 
\end{itemize}
The family of such curves is irreducible of dimension~$4$.
\end{corollary}

\begin{remark}\label{rem:hypgen3}
If a genus~$3$ curve $C$ is both bielliptic and hyperelliptic then $C$
has to be an \'etale double cover of a genus~$2$ curve, say $T$: see
\cite[Proposition~5]{Bor}. Then there exists a polarised isogeny
$\pi\colon A\to JT$ such that $C=\pi^{-1}(T)$. The converse also holds:
any \'etale double cover of a genus~$2$ curve is
hyperelliptic and bielliptic \cite[page~346]{Mumfordprym}. In
particular, as the number of non-zero $2$-torsion points in the kernel
of the polarisation on $A$ is three, there are exactly three
hyperelliptic curves in the linear system of a $(1, 2)$-polarising
line bundle on a very general abelian surface \cite[Proposition~6]{Bor}.
\end{remark}

\section{Genus four curves and $(1,3)$ theta divisors}\label{sect:g=4}
Now we focus on genus~$4$ curves. They have special
properties because, firstly, they can be embedded only in $(1,3)$
polarised surfaces and, secondly, the complementary abelian subvariety
is also a $(1,3)$-polarised surface.  Using the additional symmetry,
we have the following theorem.
\begin{theorem}\label{thm:13equiv}
Let $C$ be a smooth genus $4$ curve. The following conditions are equivalent.
\begin{itemize}
\item[(i)] $JC\in\Is^4_{(1,3)}$.
\item[(ii)] $C$ can be embedded in an abelian surface $A$.
\item[(iii)] $C$ has two non-isomorphic embeddings into abelian surfaces $A_1$ and $A_2$.
\item[(iv)] $JC$ contains two complementary abelian surfaces $\hat{A}$ and
  $\hat{B}$ with restricted polarisation of type $(1,3)$.\qed
\end{itemize}
\end{theorem}
We say that two embeddings $\varphi_i\colon C\to A_i$  are isomorphic
if there is a polarised isomorphism $\psi\colon A_1\to A_2$ such that
$\varphi_2=\psi\varphi_1$. 

There is a well-known family of curves that fulfill
the conditions of Theorem~\ref{thm:13equiv}. We will follow results
from~\cite{R}.  Let $\pi\colon C\To C'$ be an \'etale triple cyclic
cover of a genus~$2$ curve, defined by a $3$-torsion point $\eta\in
JC'$. Choose $\sigma$, a lift to $C$ of the hyperelliptic involution of
$C'$. Then $E=C/\sigma$ is an elliptic curve. Moreover, $A=JC/\eta$ is
embedded in $JC$, and by \cite[Thm.~1]{R}, the Prym surface $P(C/C')$
is $E\times E$ with a polarisation given by
$\cO(E\times\{0\}\cup\{0\}\times E\cup \Delta)$, where $\Delta$ is the
diagonal: see also~\cite{BL2}. Moreover, $A$ and $P(C/C')$ are
complementary to each other and of type~$(1,3)$.

\subsection{Theta divisor}\label{subsect:theta}
Let $A=A_Z$ be an abelian surface with a $(1,3)$ polarisation
$H$, as in Section~\ref{sect:intro}. There is a unique odd section (up
to scalar) of $L_0$, given in terms of classical theta functions by
$\theta_A=\theta[^{-\omega}_{\ 0}]-\theta[^{\omega}_0]$ or, in terms
of canonical theta functions, by $\theta_A=\theta^0_{-\omega}-\theta^0_{\omega}$.

\begin{definition}\label{def:13thetadivisor}
The $(1,3)$ theta divisor of $A$ is the curve $C_A=(\theta_A=0)$ in $A$.
\end{definition}

We made a few choices to define $C_A$, but it is well-defined up to
translation, as the following (essentially \cite[Prop. 4.6.5]{BL}) shows.

\begin{proposition}\label{prop:13thetadivisorunique}
Suppose $L$ is a symmetric polarising line bundle of type $(1,3)$ and
characteristic $c = c_1+c_2$ with respect to some decomposition on an
abelian surface~$A$, so that $K(L)_1 = \{0,\eta,-\eta\}$, for some
$\eta\in A[3]$. Let $h^0_{\pm}$ be the dimensions of the $(\pm
1)$-eigenspaces of the $(-1)$ action on $H^0(L)$. Then
\begin{itemize}
\item if $c$ is even then $h^0_+= 2,\ h^0_-=1$;
\item if $c$ is odd then $h^0_+=1,\ h^0_-=2$.
\end{itemize}
In both cases
$$
\theta_{A,L} := \theta^c_{\eta}-\exp(4\pi
i\Ima(H)(\eta,c_2))\theta^c_{-\eta}
$$ 
generates the $1$-dimensional eigenspace, and so for every
characteristic~$c$ we have
$$
(\theta_A = 0) = t_c^*(\theta_{A,L} = 0).
$$
\end{proposition}

The following lemma gives us a few basic properties of $C_A$.
\begin{lemma}\label{lem:baspr} 
Let $A$ be a $(1, 3)$-polarised surface and $C_A$ be the $(1,
3)$ theta divisor. Then: 
\begin{itemize}
\item[(i)] $C_A$ is of arithmetic genus $4$.
\item[(ii)] $C_A$ passes through at least ten $2$-torsion points on $A$.
\item[(iii)] If $A$ is a general abelian surface, then $C_A$ is smooth. 
\item[(iv)] If $C_A$ is smooth then it is hyperelliptic.
\item[(v)] For $A = E\times F$ with the product polarisation
  $\cO_E(1)\boxtimes \cO_F(3)$, where $E,\, F$ are elliptic curves, the
  curve $C_A$ is reducible and consists of one copy of~$F$ and three
  copies of~$E$.
\end{itemize}
\end{lemma}
\begin{proof}
(i) follows from adjunction and Riemann-Roch, (ii) is a consequence of
  Proposition~\ref{prop:tenpoints}, and (iii) follows from the work of
  Andreotti and Mayer: see \cite[Prop~6]{AM} or \cite[Ch.~6.4]{ACGH}
  for details.
  
For~(iv), the involution $(-1)$ on $A$ restricts to $C_A$ and the
quotient $C'=C_A/\pm 1$ is a smooth curve. The quotient map $C_A\to
C'$ is ramified at the $2$-torsion points of $A$ lying on $C_A$, of
which there are $b\ge 10$ by~(ii). The Hurwitz formula now gives
$$
2g(C_A) - 2 = 2(2g(C') - 2) + b
$$
and since $g(C_A)=4$ and $b \geq 10$, the only possibility is $g(C') =
0$ and $b = 10$.

In the situation of~(v), the matrix $Z$ can be chosen to be diagonal,
so the theta function is of the form $\theta(v_1, v_2) =
f(v_1)g(v_2)$, where $f \in H^0(\cO_E(1))$ and $g\in
H^0(\cO_F(3))$. Therefore, $f$ has exactly one zero and $g$ has three
zeros, which gives the assertion.
\end{proof}

\subsection{Product of elliptic curves}
If $(A,H)$ is a product, we can compute more details. Let
$$
E = \CC/\tau_1\ZZ+\ZZ,\ F = \CC/\tau_2\ZZ+\ZZ,\ 
\Lambda=\left[\begin{array}{cccc} \tau_1& 0& 1& 0\\ 0&\tau_2& 0&
    3\end{array}\right],\ A = \CC^2/\Lambda.
$$ 
Then $A = E\times F$,
with the product polarisation. We can take a standard decomposition
$$
\CC^2 =\left[\begin{array}{cc} \tau_1&
    0\\ 0&\tau_2\end{array}\right]\RR^2 + \RR^2
$$ 
and write theta functions explicitly: for $v=(v_1,v_2)\in\CC^2$ we have
\[
\Th{\pm\omega}{0}(v,Z)
= \sum_{l_1,l_2\in\ZZ} \exp(\pi il_1^2\tau_1
+ \pi i(l_2 \pm\frac{1}{3})^2\tau_2 
+2\pi i l_1v_1
+ 2\pi iv_2(l_2 +\frac{1}{3})).
\]
For the computations, let us denote
\begin{align*}
a_l &= \exp(\pi il^2\tau_1 + 2\pi i v_1l),\\
b_l^\pm &= \exp(\pi i(l\pm\frac{1}{3})^2\tau_2 + 2\pi i v_2(l\pm\frac{1}{3})).
\end{align*}
Then, since the series converge absolutely
$$
\theta_A = \theta[^{\omega}_0]-\theta[^{-\omega}_{\ 0}]
=\sum_{l_1,l_2}a_{l_1}b^+_{l_2}-\sum_{l_1,l_2}a_{l_1}b^-_{l_2}
=\sum_{l_1}a_{l_1}(\sum_{l_2} b^+_{l_2}-\sum_{l_2} b^-_{l_2}).
$$
For $v_1 = \frac{1}{2} + \frac{1}{2}\tau_1$ we have
\begin{align*}
a_l &= \exp(\pi il^2\tau_1+\pi i l\tau_1 + \pi il)\\
    &= \exp(\pi i(l+\frac{1}{2})^2\tau_1-\frac{1}{4}\pi i\tau_1 + \pi
il)\\
    &= (-1)^l\exp(\pi i(l+\frac{1}{2})^2-\frac{1}{4}\pi i\tau_1).
\end{align*}
Now,
$a_l = -a_{-l-1}$, so $\sum_l a_l = 0$ and therefore for any $v_2$ we
find $\theta_A(\frac{1}{2} + \frac{1}{2}\tau_1,v_2)=0$.  The image of
this component of $\{v\mid \theta_A(v)=0\}$ in $A$ is a curve isomorphic to $F$.

Similar computations can be carried out for $v_2=0,\ v_2=\frac{3}{2}$
and $v_2=\frac{1}{2}\tau_2$. For $v_2=0$ and $v_2=\frac{3}{2}$, we
have $b^+_l=b^-_l$.  For $v_2=\frac{1}{2}\tau_2$ we have
$b^+_l=b^-_{-l-1}$. In all three cases we get $\sum_l b_l^+=\sum_l b_l^-$.
The images in $A$ of those zeros are isomorphic to $E$, so by
Lemma~\ref{lem:baspr}, we know we have found all zeros of $\theta_A$.

This gives us some more information about $C_A$ in general.
\begin{proposition}\label{prop:notlowergenus}
For a general $(A,H)$ of type $(1,3)$, the $(1,3)$ theta divisor 
$C_A$ is not a finite cover of a curve of lower positive genus.
\end{proposition}
\begin{proof}
If not, then $3E\cup F$ would be specialisations and so
they would be covers of possibly singular curves. The only nontrivial
case is to show that $3E\cup F$ is not a triple cover of a genus~$2$
curve. Then the intersection points of $F$ and the copies of $E$ would be
$3$-torsion points on $F$, but they are in fact $2$-torsion points.
\end{proof}

In \cite{BL2}, Ch.~Birkenhake and H.~Lange consider $E\times E$ with
the polarisation given by $E\times\{0\}\,\cup\,\{0\}\times E\,\cup\,
\pm\Delta$, where $\pm\Delta$ is the diagonal or antidiagonal
respectively. They prove that both polarisations are of type $(1,3)$
and that they are dual to each other. As $C_A$ is characterised as the
only symmetric divisor that passes through ten $2$-torsion points with
odd multiplicity, it is easy to see that in this case both of
$E\times\{0\}\,\cup\,\{0\}\times E\,\cup\, \pm\Delta$ are $(1,3)$ theta
divisors (for different polarised surfaces).

\subsection{Jacobian of $C_A$}

\begin{proposition}\label{prop:lemat} 
Let $C$ be a smooth hyperelliptic genus~$4$ curve. Then $JC$ contains
a $(1, 3)$-polarised surface $\hat{A}$ if and only if $C$ can be
embedded into $A$ as the $(1, 3)$ theta divisor.
\end{proposition}
\begin{proof}
In view of Theorem~\ref{thm:13equiv} and Lemma~\ref{lem:baspr}, it
remains to prove that if $C$ is hyperelliptic with $JC \in
\Is^4_{(1,3)}$ then $C = C_A$ is the $(1, 3)$ theta divisor
of~$A$. For $i\colon\hat{A}\to JC$ the inclusion of a surface with
restricted polarisation of type $(1,3)$, we write $\hat{\imath} \colon
JC \to A$ for the dual map, $\iota$ for the hyperelliptic involution,
and $\alpha$ for the Abel-Jacobi map, and identify $C$ with
$\alpha(C)\subset JC$.  Then $(-1)^*C = C$ is a symmetric curve
because $\alpha(P) = -\alpha(\iota(P))$. Therefore the image
$\hat{\imath}(C)$ is also symmetric in $A$. Note that by
Proposition~\ref{prop:indpol}, $\hat{\imath}(C)$ is isomorphic to $C$.

Let $\pi_{JC} \colon JC \To JC/(-1)$ be the Kummer map. Then
$\pi_{JC}(C)$ is isomorphic to $\PP^1$, because $(-1)|_C = \iota$. Let
$\pi_A$ be the Kummer map of $A$. Because $\hat{\imath}$ is a
homomorphism, it descends to a map $j$ making the diagram
\begin{equation*}
  \begin{diagram}
\label{DDD}
\dgARROWLENGTH=2em
    \node{JC}\arrow[4]{e,t}{\hat{\imath}}\arrow[2]{s,r}{\pi_{JC}}    
\node[4]{A}\arrow[2]{s,r}{\pi_{A}}\\
[2]\node{JC/(-1)}\arrow[4]{e,t}{j}\node[4]{A/(-1)}
  \end{diagram}
\end{equation*}
commute, given by $j(\pm x) = \pm \hat{\imath}(x)$.  Now,
$j(\pi_{JC}(C))$ has to be a rational curve, equal to $\pi_A
(\hat{\imath}(C))$. As $\pi_A$ is a $2$-to-$1$ map, by the Hurwitz
formula it has to be branched in ten points. The only possible branch
points are $2$-torsion points, so $\hat{\imath}(C)$ has to go through
ten $2$-torsion points of $A$. By Proposition~\ref{prop:tenpoints},
$\hat{\imath}(C) \cong C$ is the zero locus of an odd global section,
which finishes the proof.
\end{proof}
From Proposition~\ref{prop:lemat} we get the following.
\begin{theorem}\label{thm:jednahiper}
A general $(1,3)$-polarised surface contains exactly one hyperelliptic
curve up to translation.
\end{theorem}

\begin{remark}\label{rem:BOPY}
The authors of \cite{BOPY} use Gromov-Witten theory of K3 surfaces to
count hyperelliptic curves of arithmetic genus $g$ on a fixed linear
system of a $(1,d)$-polarised abelian surface: see
also~\cite{Ro}. These numbers, called $h^{A,FLS}_{g,\beta}$, are
presented in \cite[Table~1]{BOPY}. The numbers on the diagonal of that
table, where $d=g+1$, are of particular interest because then the
curves have to be smooth. These numbers are non-zero only for
$g=2,\,3,\,4$ and~$5$, and we can explain them as follows.

For $g=2$, the surfaces are principally polarised so the number
$h^{A,FLS}_{g,\beta}=1$ corresponds to the curve embedded in its
Jacobian.

For $g=3$, there are six hyperelliptic curves according
to~\cite{BOPY}, yet by Remark~\ref{rem:hypgen3} there are only three
such curves up to translation. The difference follows from the fact
that we can translate any curve by an element of the kernel of the
polarisation and still stay in the fixed linear system. In this case,
although the kernel has order four, we get only two copies of each
curve in this way, because both copies are invariant under translation
by a $2$-torsion point that defines the polarised isogeny $\pi\colon
A\to JT$ described in Remark~\ref{rem:hypgen3}.

For $g=4$, the order of the kernel of the polarisation is~$9$. By
Proposition \ref{prop:notlowergenus}, we know we get exactly~$9$
translates of the $(1,3)$ theta divisor.

The number for $g=5$ is explained by forthcoming work of the first
author and A.~Ortega~\cite{BoOr}, in which it is shown that such a
curve is unique up to translation. The situation is similar to the
$g=3$ case: the kernel has order~$16$, but hyperelliptic curves are
invariant under a subgroup of order~$4$.
\end{remark}

Now we will prove the main result of this paper.
\begin{theorem}\label{thm:dualpair}
Let $C$ be a smooth hyperelliptic genus~$4$ curve. Then $JC$ contains a
pair of complementary $(1,3)$-polarised surfaces $A$ and $\hat{A}$ if
and only if $C$ can be embedded in $A$ as the $(1,3)$ theta divisor.
\end{theorem}
\begin{proof}
One implication follows from Proposition~\ref{prop:lemat}.

Now, recall from \cite{BN} that a (desingularised) Kummer surface of a
general $(1,3)$-polarised abelian surface $A$, denoted $\Km(A)$, can
be embedded as a $H_{22}$-invariant quartic surface in $\PP^3$
containing a $(32)_{(10)}$ configuration of lines. That means that
there are~$16$ pairwise disjoint lines, called odd lines, and $16$
pairwise disjoint lines, called even lines, such that every odd line
intersects exactly $10$ even lines and every even line intersects
exactly $10$ odd ones.  It is easy to see that the two groups of lines
are exactly images of hyperelliptic curves $C_A$ translated by
$2$-torsion points, and $2$-torsion points blown-up.

Moreover, $\Km(A)=\Km(\hat{A})$ and the surfaces can be reconstructed
by choosing which set of $16$ disjoint lines are to be blown down.
Now, the fact that $C_A\cong C_{\hat{A}}$ follows from the fact that
there exists an involution in the extended Heisenberg group that
changes odd lines to even ones: see~\cite[Section~4]{BN}. Restricting
the involution to a line with $10$ points marked, we obtain the
desired isomorphism. As $C_A\cong C_{\hat{A}}$, we have
$JC_A=JC_{\hat{A}}$ and, as $A$ can be chosen to be simple, we get the
assertion.
\end{proof}
\begin{remark}\label{rem:dualdiagrams}
Theorem \ref{thm:dualpair} shows that the kernel $K^0$ that occurs in
Diagram~\eqref{Diag1} is in fact isomorphic to $\hat{A}$. We can draw
two diagrams for embeddings in $A$ and $\hat{A}$ respectively.
\begin{multicols}{2}
\begin{equation}\label{Diag2}
  \begin{diagram}
\dgARROWLENGTH=2.5em
    \node{C}\arrow[4]{e,t}{f_{C,A}}\arrow[2]{se,r}{\alpha_O}
    \node[4]{A}\\[2]\node[3]{JC}
\arrow [2]{ne,r}{f_A}
    \\[2]
    \node[1]{\hat{A}}\arrow[2]{ne,r}{k_{\hat{A}}}  
  \end{diagram}
\end{equation}

\begin{equation*}\label{Diag3}
  \begin{diagram}
\dgARROWLENGTH=2.5em
    \node{C}\arrow[4]{e,t}{f_{C,\hat{A}}}\arrow[2]{se,r}{\alpha_O}
    \node[4]{\hat{A}}\\[2]\node[3]{JC}
\arrow [2]{ne,r}{f_{\hat{A}}}
    \\[2]
    \node[1]{A}\arrow[2]{ne,r}{k_A}  
  \end{diagram}
\end{equation*}
\end{multicols}
From the construction, we see that $k_A=\hat{f_A}$ and
$k_{\hat{A}}=\hat{f}_{\hat{A}}$, so the diagrams are dual to each other.
In particular, the rational map 
$$
\Phi\colon \cA_{(1,3)}\ratmap \cA_{(1,3)},\text{ given by }\hat{A}\To K^0,
$$ 
defined by the construction and Diagram~\eqref{Diag1}, is the
dualisation in the moduli space.
\end{remark}

This construction raises several questions. We should like more
information about the locus $\Is^4_{1,3}$ and more generally about the
loci $\Is^g_D$ in the moduli spaces of abelian varieties. Some answers
are given in~\cite{Bor}. Alternatively, one could look at the loci in
$\cM_4$ or $\Bar\cM_4$ determined by the curves $C_A$, and ask for the
class of the corresponding cycle in cohomology or the Chow
ring. One could also ask for a description of the sets of
ten points in $\PP^1$ that form branch loci of the hyperelliptic
involutions of the $C_A$. 

One could also ask about the genus~$4$ curves whose Jacobians contain
$(1,3)$-polarised surfaces. What are their properties? The question
seems to be interesting, because on the side of Jacobians, the locus
is naturally defined, whereas on $\cM_4$, it contains disjoint subloci
of hyperelliptic curves and covers of genus~$2$ curves. Those loci
seems to be completely differently treated in the moduli of curves
theory.

Finally, we ask whether it is possible to find genus~$4$ curves whose
Jacobian contains a pair of complementary abelian subvarieties that
are isomorphic to each other. If so, what are their properties?

\end{document}